\documentclass[reqno]{amsart}

\usepackage{amssymb}
\usepackage{graphicx}
\usepackage{amscd}
\usepackage[hidelinks]{hyperref}
\usepackage{color}
\usepackage{float}
\usepackage{graphics,amsmath,amssymb}
\usepackage{amsthm}
\usepackage{amsfonts}
\usepackage{latexsym}
\usepackage{epsf}
\usepackage{xifthen}
\usepackage{mathrsfs}
\usepackage{dsfont}
\usepackage{makecell}
\usepackage[FIGTOPCAP]{subfigure}
\usepackage{amsmath}
\allowdisplaybreaks[4]
\usepackage{listings}
\usepackage{etoolbox}
\usepackage{fancyhdr}
\usepackage{pdflscape}
\usepackage[title,toc,titletoc]{appendix}
\usepackage{enumitem}
\usepackage[noadjust]{cite}
\usepackage{forest}
\usepackage{tikz}
\usepackage{ytableau}

 \newtheoremstyle{mytheorem}% name % cf. thmtest.tex of AMSLaTeX
 {3pt}%      Space above
 {3pt}%      Space below
 {\slshape}% Body font
 {}%         Indent amount (empty = no indent,
 {\bfseries}% Thm head font
 {.}%        Punctuation after thm head
 { }%        Space after thm head: " " = normal interword space;
 {}%         Thm head spec (can be left empty, meaning `normal')

\numberwithin{equation}{section}

\theoremstyle{theorem}
\newtheorem{theorem}{Theorem}[section]
\newtheorem*{theorem*}{Theorem}

\newtheorem{lemma}[theorem]{Lemma}

\newtheorem{innercustomgeneric}{\customgenericname}
\providecommand{\customgenericname}{}
\newcommand{\newcustomtheorem}[2]{%
	\newenvironment{#1}[1]
	{%
		\renewcommand\customgenericname{#2}%
		\renewcommand\theinnercustomgeneric{##1}%
		\innercustomgeneric
	}
	{\endinnercustomgeneric}
}
\newcustomtheorem{ctheorem}{Theorem}

\theoremstyle{definition}

\newtheorem*{example*}{Example}
\newtheorem{conjecture}{Conjecture}[section]

\theoremstyle{remark}
\newtheorem{remark}{Remark}[section]
\newtheorem*{remark*}{Remark}
\newtheorem*{remarks*}{Remarks}

\newtheoremstyle{named}{}{}{\itshape}{}{\bfseries}{.}{.5em}{#1\thmnote{ #3}}
\theoremstyle{named}

\newcommand{\Keywords}[1]{\ifthenelse{\isempty{#1}}{}{\smallskip \smallskip \noindent \textbf{Keywords}. #1}}
\newcommand{\MSC}[2][2020]{\ifthenelse{\isempty{#2}}{}{\smallskip \smallskip \noindent \textbf{#1MSC}. #2}}
\newcommand{\abstractnote}[1]{\ifthenelse{\isempty{#1}}{}{\smallskip \smallskip \noindent \textsuperscript{\dag}#1}}

\makeatletter
\def\specialsection{\@startsection{section}{1}%
  \z@{\linespacing\@plus\linespacing}{.5\linespacing}%
  {\normalfont}}% NEW
\def\section{\@startsection{section}{1}%
  \z@{.7\linespacing\@plus\linespacing}{.5\linespacing}%
  {\normalfont\scshape}}% NEW
\patchcmd{\@settitle}{\uppercasenonmath\@title}{\Large\boldmath}{}{}
\patchcmd{\@settitle}{\begin{center}}{\begin{flushleft}}{}{}
\patchcmd{\@settitle}{\end{center}}{\end{flushleft}}{}{}
\patchcmd{\@setauthors}{\MakeUppercase}{\normalsize}{}{}
\patchcmd{\@setauthors}{\centering}{\raggedright}{}{}
\patchcmd{\section}{\scshape}{\large\bfseries\boldmath}{}{}
\patchcmd{\subsection}{\bfseries}{\bfseries\boldmath}{}{}
\renewcommand{\@secnumfont}{\bfseries}
\patchcmd{\@startsection}{\@afterindenttrue}{\@afterindentfalse}{}{}
\patchcmd{\abstract}{\leftmargin3pc}{\leftmargin1pc}{}{}

\def\maketitle{\par
  \@topnum\z@ % this prevents figures from falling at the top of page 1
  \@setcopyright
  \thispagestyle{empty}% this sets first page specifications
  \ifx\@empty\shortauthors \let\shortauthors\shorttitle
  \else \andify\shortauthors
  \fi
  \@maketitle@hook
  \begingroup
  \@maketitle
  \toks@\@xp{\shortauthors}\@temptokena\@xp{\shorttitle}%
  \toks4{\def\\{ \ignorespaces}}% defend against questionable usage
  \edef\@tempa{%
    \@nx\markboth{\the\toks4
      \@nx\MakeUppercase{\the\toks@}}{\the\@temptokena}}%
  \@tempa
  \endgroup
  \c@footnote\z@
  \@cleartopmattertags
}
\makeatother

\title{On a congruence involving harmonic series and Bernoulli numbers}

\author[S. Chern]{Shane Chern}
\address{Department of Mathematics and Statistics, Dalhousie University, Halifax, Nova Scotia, B3H 4R2, Canada}
\email{chenxiaohang92@gmail.com}

\date{}

\begin{document}

\maketitle

\begin{abstract}

In 2003, Zhao discovered a curious congruence involving harmonic series and Bernoulli numbers: for any odd prime $p$,
$$\sum_{\substack{i,j,k\ge 1\\\gcd(ijk,p)=1\\i+j+k=p}}\frac{1}{ijk}\equiv -2B_{p-3} \pmod{p},$$
where $B_n$ is the $n$-th Bernoulli number. This congruence was generalized by Wang and Cai in 2014, and Cai, Shen and Jia in 2017 by replacing the odd prime $p$ in the summation and modulus with an odd prime power, and a product of two odd prime powers, respectively. In particular, Cai, Shen and Jia proposed a conjectural congruence: for any positive integer $n$ with an odd prime factor $p$ such that $p^r \parallel n$ where $r\ge 1$,
$$\sum_{\substack{i,j,k\ge 1\\\gcd(ijk,n)=1\\i+j+k=n}}\frac{1}{ijk}\equiv -2B_{p-3}\cdot \frac{n}{p}\cdot \prod_{\substack{\text{prime $q\mid n$}\\q\ne p}}\left(1-\frac{2}{q}\right)\left(1-\frac{1}{q^3}\right) \pmod{p^r}.$$
In this paper, we establish the following generalization of their conjecture: for any positive integer $n$ with an odd prime factor $p$ such that $p^r \parallel n$ where $r\ge 1$,
$$\begin{aligned}
\sum_{\substack{i,j,k\ge 1\\\gcd(ijk,n)=1\\a_1 i+a_2 j+a_3 k=An}}\frac{1}{ijk}&\equiv -2B_{p-3}\cdot \frac{n}{p}\cdot \frac{Ag^3}{3}\left(\frac{1}{a_1^2 g_1^2}+\frac{1}{a_2^2 g_2^2}+\frac{1}{a_3^2 g_3^2}\right)\\
&\quad\times \prod_{\substack{\text{prime $q\mid n$}\\q\ne p}}\left(1-\frac{2}{q}\right)\left(1-\frac{1}{q^3}\right) \pmod{p^r},
\end{aligned}$$
where $a_1$, $a_2$ and $a_3$ are positive integers coprime to $p$, and $A$ is a positive common multiple of $a_1$, $a_2$ and $a_3$. Also, $g_1=\gcd(a_2,a_3)$, $g_2=\gcd(a_3,a_1)$, $g_3=\gcd(a_1,a_2)$ and $g=\gcd(a_1,a_2,a_3)$.


\MSC{11A07, 11A41.}
\end{abstract}

\section{Introduction}

In 2003, Zhao \cite{Zha2003} discovered a curious congruence involving harmonic series and Bernoulli numbers: for any odd prime $p$,
\begin{align*}
\sum_{\substack{i,j,k\ge 1\\\gcd(ijk,p)=1\\i+j+k=p}}\frac{1}{ijk}\equiv -2B_{p-3} \pmod{p},
\end{align*}
where $B_n$ is the $n$-th Bernoulli number defined by the exponential generating function
\begin{align*}
\frac{z}{e^z-1}=\sum_{n\ge 0}\frac{B_n z^n}{n!}.
\end{align*}
Here, $B_0=1$. Also, a theorem of von Staudt and Clausen \cite[p.~275]{Apo1976} asserts that for any positive integer $n$,
\begin{align*}
B_{2n}+\sum_{(p-1)\mid 2n}\frac{1}{p}\in\mathbb{Z},
\end{align*}
where the summation runs over primes $p$ such that $p - 1$ divides $2n$. Therefore, $B_{p-3}$ is a $p$-adic integer for all odd primes $p$.

Zhao's proof based on partial sums of a multiple zeta value series was published in \cite{Zha2007} while Ji \cite{Ji2005} presented an elementary proof slightly earlier. Meanwhile, this congruence can be refined along different directions. For example, Zhou and Cai \cite{ZC2007} generalized the triple summation to an $n$-folder summation. On the other hand, Wang and Cai \cite{WC2014}, and Cai, Shen and Jia \cite{CSJ2017} replaced the odd prime $p$ in the summation and modulus with an odd prime power, and a product of two odd prime powers, respectively. For other related papers, see \cite{MTWZ2017,SC2018,Wan2015,XC2010}.

In \cite{CSJ2017}, Cai, Shen and Jia also proposed the following conjecture.
\begin{conjecture}
	For any positive integer $n$ with an odd prime factor $p$ such that $p^r \parallel n$ (that is, $p^r\mid n$ and $p^{r+1}\nmid n$) for a certain positive integer $r$,
	\begin{align*}
	\sum_{\substack{i,j,k\ge 1\\\gcd(ijk,n)=1\\i+j+k=n}}\frac{1}{ijk}\equiv -2B_{p-3}\cdot \frac{n}{p}\cdot \prod_{\substack{\text{prime $q\mid n$}\\q\ne p}}\left(1-\frac{2}{q}\right)\left(1-\frac{1}{q^3}\right) \pmod{p^r}.
	\end{align*}
\end{conjecture}

The object of this paper is a generalization of their conjecture.

\begin{theorem}\label{th:main}
	Let $a_1$, $a_2$ and $a_3$ be positive integers and let $A$ be a positive common multiple of $a_1$, $a_2$ and $a_3$. Let $g_1=\gcd(a_2,a_3)$, $g_2=\gcd(a_3,a_1)$, $g_3=\gcd(a_1,a_2)$ and $g=\gcd(a_1,a_2,a_3)$. Then for any positive integer $n$ with an odd prime factor $p$ such that $p^r\parallel n$ for a certain positive integer $r$, and $p$ does not divide $a_1$, $a_2$ and $a_3$,
	\begin{align}
	\sum_{\substack{i,j,k\ge 1\\\gcd(ijk,n)=1\\a_1 i+a_2 j+a_3 k=An}}\frac{1}{ijk}&\equiv -2B_{p-3}\cdot \frac{n}{p}\cdot \frac{Ag^3}{3}\left(\frac{1}{a_1^2 g_1^2}+\frac{1}{a_2^2 g_2^2}+\frac{1}{a_3^2 g_3^2}\right)\notag\\
	&\quad\times \prod_{\substack{\textup{prime $q\mid n$}\\q\ne p}}\left(1-\frac{2}{q}\right)\left(1-\frac{1}{q^3}\right) \pmod{p^r}.
	\end{align}
\end{theorem}

\begin{remark}
	Notice that in Theorem \ref{th:main}, we always have that
	$$\frac{Ag^3}{3}\left(\frac{1}{a_1^2 g_1^2}+\frac{1}{a_2^2 g_2^2}+\frac{1}{a_3^2 g_3^2}\right)=\frac{Ag^3}{a_1^2a_2^2a_3^2g_1^2g_2^2g_3^2}\cdot \frac{a_1^2a_2^2g_1^2g_2^2+a_2^2a_3^2g_2^2g_3^2+a_3^2a_1^2g_3^2g_1^2}{3}$$
	is a $p$-adic integer. Recall that $p$ does not divide $a_1$, $a_2$ and $a_3$ (and therefore $g_1$, $g_2$ and $g_3$). When $p\ge 5$, the above claim is trivial. When $p=3$, we know that $x^2\equiv 1 \pmod{3}$ for any $x$ not divisible by $3$, and thus, $3$ divides $a_1^2a_2^2g_1^2g_2^2+a_2^2a_3^2g_2^2g_3^2+a_3^2a_1^2g_3^2g_1^2$, which confirms the claim.
\end{remark}

\section{Lemmas}

\begin{lemma}\label{le:arith-prog-p-r}
	Let $p$ be an odd prime and $r$ be a positive integer. Let $u$ be a positive integer. If $m$ is a positive integer not divisible by $p$, then for any integer $\ell$,
	\begin{align}
	\sum_{\substack{1\le i\le u m p^r\\i\equiv \ell \ (\operatorname{mod}\;m)\\\gcd(i,p)=1}}\frac{1}{i}\equiv 0 \pmod{p^r}.
	\end{align}
\end{lemma}

\begin{proof}
	We start by noticing that
	\begin{align*}
	\{i\ (\operatorname{mod}\;p^r):1\le i\le m p^r\text{ and }i\equiv \ell \ (\operatorname{mod}\;m)\}
	\end{align*}
	equals $\mathbb{Z}/p^r\mathbb{Z}$. This is because if $i_1$ and $i_2$ simultaneously satisfy $i_1\equiv i_2 \pmod{p^r}$ and $i_1\equiv i_2\equiv \ell \pmod{m}$, then they differ by a multiple of $mp^r$ since $p$ is coprime to $m$. Therefore, modulo $p^r$,
	\begin{align*}
	\sum_{\substack{1\le i\le u m p^r\\i\equiv \ell \ (\operatorname{mod}\;m)\\\gcd(i,p)=1}}\frac{1}{i}\equiv u \sum_{\substack{1\le i\le m p^r\\i\equiv \ell \ (\operatorname{mod}\;m)\\\gcd(i,p)=1}}\frac{1}{i}\equiv u \sum_{\substack{1\le i\le p^r\\\gcd(i,p)=1}}\frac{1}{i}= \frac{u}{2}\sum_{\substack{1\le i\le p^r-1\\\gcd(i,p)=1}}\left(\frac{1}{i}+\frac{1}{p^r-i}\right)\equiv 0.
	\end{align*}
	This is our desired result.
\end{proof}

\begin{lemma}\label{le:mod-c}
	Let $p$ be an odd prime and $r$ be a positive integer. Let $u$ and $v$ be positive integers. If $c$ is a positive integer not divisible by $p$, then for any positive integers $a$ and $b$ coprime to $c$,
	\begin{align}
	\sum_{\substack{1\le i\le u cp^r\\1\le j\le v cp^r\\ai\equiv bj \ (\operatorname{mod}\;c)\\\gcd(ij,p)=1}}\frac{1}{ij}\equiv 0 \pmod{p^{2r}}.
	\end{align}
\end{lemma}

\begin{proof}
	We have
	\begin{align*}
	\sum_{\substack{1\le i\le u cp^r\\1\le j\le v cp^r\\ai\equiv bj \ (\operatorname{mod}\;c)\\\gcd(ij,p)=1}} = \sum_{\ell=1}^c \sum_{\substack{1\le i\le u cp^r\\ai\equiv \ell \ (\operatorname{mod}\;c)\\\gcd(i,p)=1}}\frac{1}{i}\sum_{\substack{1\le j\le v cp^r\\bj\equiv \ell \ (\operatorname{mod}\;c)\\\gcd(j,p)=1}}\frac{1}{j}.
	\end{align*}
	Since $a$ is coprime to $c$, we know that it is invertible modulo $c$. Let $\overline{a}$ be such that $\overline{a}a\equiv 1\pmod{c}$. Then
	\begin{align*}
	\sum_{\substack{1\le i\le u cp^r\\ai\equiv \ell \ (\operatorname{mod}\;c)\\\gcd(i,p)=1}}\frac{1}{i} = \sum_{\substack{1\le i\le u cp^r\\i\equiv \overline{a}\ell \ (\operatorname{mod}\;c)\\\gcd(i,p)=1}}\frac{1}{i}\equiv 0 \pmod{p^r},
	\end{align*}
	with Lemma \ref{le:arith-prog-p-r} applied. Similarly, since $b$ and $c$ are coprime, we have
	\begin{align*}
	\sum_{\substack{1\le j\le v cp^r\\bj\equiv \ell \ (\operatorname{mod}\;c)\\\gcd(j,p)=1}}\frac{1}{j}\equiv 0 \pmod{p^r}.
	\end{align*}
	Therefore, the lemma follows.
\end{proof}

\begin{lemma}\label{le:U}
	Let $p$ be an odd prime and $r$ be a positive integer. For any positive integer $u$ and integer $s$ not divisible by $p$, we define
	\begin{align*}
	U(s;u,p^r):=\sum_{k=0}^{up^{r-1}-1}\frac{1}{kp+s}.
	\end{align*}
	Then
	\begin{enumerate}[label={\textup{(\alph*).}},leftmargin=*,labelsep=0cm,align=left]
		\item 
		$U(s;u,p^2)\equiv pU(s;u,p) \pmod{p^2}$ and for $r\ge 2$, $U(s;u,p^{r+1})\equiv pU(s;u,p^r)\pmod{p^{r+2}}$;
		
		\item 
		for $r\ge 1$, $U(s;u,p^r)\equiv 0\pmod{p^{r-1}}$;
		
		\item 
		for $r\ge 2$, $U(s;u,p^{r+1})U(t;v,p^{r+1})\equiv p^2U(s;u,p^{r})U(t;v,p^{r}) \pmod{p^{2r+2}}$;
		
		\item 
		for $r\ge 2$,
		\begin{align*}
		&U(s;u,p^r)U(t;v,p^r)\\
		&\quad \equiv\begin{cases}
		uvs^{-1}t^{-1}p^{2r-2}+\frac{uvs^{-1}t^{-3}(t+1)+uvs^{-3}(s+1)t^{-1}}{2}p^{2r-1} & \text{if $p=3$}\\[8pt]
		uvs^{-1}t^{-1}p^{2r-2}+\frac{uvs^{-1}t^{-2}+uvs^{-2}t^{-1}}{2}p^{2r-1} & \text{if $p\ge 5$}
		\end{cases}  \pmod{p^{2r}}.
		\end{align*}
	\end{enumerate}
\end{lemma}

\begin{proof}
	We follow the proof of \cite[Lemma 3]{WC2014}.
	
	(a). For $r\ge 1$,
	\begin{align*}
	U(s;u,p^{r+1})-pU(s;u,p^r)&=\sum_{k=0}^{up^r-1}\frac{1}{kp+s}-p\sum_{k=0}^{up^{r-1}-1}\frac{1}{kp+s}\\
	&=\sum_{i=0}^{p-1}\sum_{j=0}^{up^{r-1}-1}\frac{1}{(iup^{r-1}+j)p+s}-\sum_{i=0}^{p-1}\sum_{j=0}^{up^{r-1}-1}\frac{1}{jp+s}\\
	&=-up^r\sum_{i=0}^{p-1}\sum_{j=0}^{up^{r-1}-1}\frac{i}{\big((iup^{r-1}+j)p+s\big)\big(jp+s\big)}.
	\end{align*}
	When $r=1$, we have
	\begin{align*}
	U(s;u,p^{2})-pU(s;u,p)&\equiv -up \sum_{i=0}^{p-1}\sum_{j=0}^{u-1} \frac{i}{s^2}\equiv 0 \pmod{p^2}.
	\end{align*}
	When $r\ge 2$, we have
	\begin{align*}
	U(s;u,p^{r+1})-pU(s;u,p^r)&\equiv -up^r\sum_{i=0}^{p-1}\sum_{j=0}^{up^{r-1}-1}\frac{i}{(jp+s)^2}\\
	&=-\frac{u(p-1)p^{r+1}}{2}\sum_{j=0}^{up^{r-1}-1}\frac{1}{(jp+s)^2}\\
	&\equiv 0 \pmod{p^{r+2}}.
	\end{align*}
	
	(b). This is a direct consequence of Part (a) by induction on $r$.
	
	(c). Notice that
	\begin{align*}
	&2\big(U(s;u,p^{r+1})U(t;v,p^{r+1})-p^2U(s;u,p^{r})U(t;v,p^{r})\big)\\
	&\qquad=\big(U_+(s;u)+U_+(t;v)\big)\big(U_-(s;u)+U_-(t;v)\big)\\
	&\qquad\quad-U_+(s;u)U_-(s;u)-U_+(t;v)U_-(t;v),
	\end{align*}
	where for $(\rho,\mu)=(s,v)$ or $(t,v)$,
	\begin{align*}
	U_\pm(\rho;\mu):=U(\rho;\mu,p^{r+1})\pm pU(\rho;\mu,p^{r}).
	\end{align*}
	Further, by Part (a), $U_-(\rho;\mu)\equiv 0 \pmod{p^{r+2}}$. Also, by Part (b), $U_+(\rho;\mu)\equiv 0 \pmod{p^{r}}$. The desired result therefore follows.
	
	(d). It suffices to prove the case $r=2$; the rest follows by induction on $r$ with Part (c) applied. When $r=2$, we deduce from Euler's theorem that
	\begin{align*}
	U(s;u,p^2)=\sum_{k=0}^{up-1}\frac{1}{kp+s}\equiv \sum_{k=0}^{up-1} (kp+s)^m \pmod{p^4},
	\end{align*}
	where $m=\phi(p^4)-1=p^4-p^3-1$ with $\phi$ Euler's totient function. Expanding the powers in the above and reducing with modulus $p^4$, we have
	\begin{align*}
	U(s;u,p^2)\equiv\begin{cases}
	us^mp+\dfrac{us^{m-2}(s+1)}{2}p^2-\dfrac{u^2s^{m-1}}{2}p^3 & \text{if $p=3$}\\[8pt]
	us^m p+\dfrac{us^{m-1}}{2}p^2+\dfrac{us^{m-2}-3u^2s^{m-1}}{6}p^3 & \text{if $p\ge 5$}
	\end{cases} \pmod{p^4}.
	\end{align*}
	It follows that when $p=3$,
	\begin{align*}
	&U(s;u,3^2)U(t;v,3^2)\\
	&\qquad\equiv uvs^mt^m3^2+\frac{uvs^mt^{m-2}(t+1)+uvs^{m-2}(s+1)t^m}{2}3^3\\
	&\qquad\equiv uvs^{-1}t^{-1}3^2+\frac{uvs^{-1}t^{-3}(t+1)+uvs^{-3}(s+1)t^{-1}}{2}3^3 \pmod{3^4},
	\end{align*}
	and when $p\ge 5$,
	\begin{align*}
	U(s;u,p^2)U(t;v,p^2)&\equiv uvs^mt^mp^2+\frac{uvs^mt^{m-1}+uvs^{m-1}t^m}{2}p^3\\
	&\equiv uvs^{-1}t^{-1}p^2+\frac{uvs^{-1}t^{-2}+uvs^{-2}t^{-1}}{2}p^3 \pmod{p^4},
	\end{align*}
	which is exactly what we want.
\end{proof}

\begin{lemma}\label{le:h(k)}
	Let $p$ be an odd prime. Let $c$ be a positive integer not divisible by $p$. For each positive integer $k$ with $\gcd(k,p)=1$, if we denote by $h(k)$ the unique integer such that $1\le h(k)\le cp-1$ and $h(k)\equiv ak \pmod{cp}$ where $a$ is a fixed integer that is coprime to $c$ and $p$, then
	\begin{align}
	\sum_{\substack{k=1\\\gcd(k,p)=1}}^{cp-1} \frac{1}{k\cdot h(k)} \equiv\begin{cases}
	-\dfrac{c}{a^3} & \text{if $p=3$}\\[8pt]
	\dfrac{c(a^2+1)}{3a}pB_{p-3} & \text{if $p\ge 5$}
	\end{cases} \pmod{p^2}.
	\end{align}
\end{lemma}

\begin{proof}
	Notice that
	\begin{align*}
	h(k)=-cp\left\lfloor\frac{ak}{cp}\right\rfloor+ak,
	\end{align*}
	where $\lfloor x\rfloor$ denotes the largest integer not exceeding $x$. For convenience, we write
	\begin{align*}
	\ell(k):=\left\lfloor\frac{ak}{cp}\right\rfloor.
	\end{align*}
	Let $m=\phi(p^2)-1=p^2-p-1$. It follows from Euler's theorem that
	\begin{align*}
	\sum_{\substack{k=1\\\gcd(k,p)=1}}^{cp-1} \frac{1}{k\cdot h(k)} &\equiv \sum_{\substack{k=1\\\gcd(k,p)=1}}^{cp-1} \frac{(-cp\ell(k)+ak)^m}{k}\\
	&\equiv \sum_{\substack{k=1\\\gcd(k,p)=1}}^{cp-1} \frac{(ak)^m+c(ak)^{m-1}\ell(k)p}{k}\\
	&\equiv \frac{1}{a}\sum_{\substack{k=1\\\gcd(k,p)=1}}^{cp-1}\frac{1}{k^2}+\frac{cp}{a^2}\sum_{\substack{k=1\\\gcd(k,p)=1}}^{cp-1}\frac{1}{k^3}\left\lfloor\frac{ak}{cp}\right\rfloor \pmod{p^2}.
	\end{align*}
	
	\textit{Case 1: $p\ge 5$}. We know from \cite[Theorem 1.1]{Hon2000} that
	\begin{align*}
	\sum_{\substack{k=1\\\gcd(k,p)=1}}^{cp-1}\frac{1}{k^2} \equiv \frac{2c}{3}pB_{p-3} \pmod{p^2}.
	\end{align*}
	Also, \cite[Eq.~(1)]{Por1983} implies that
	\begin{align}\label{eq:floor}
	\sum_{\substack{k=1\\\gcd(k,p)=1}}^{cp-1}\frac{1}{k^3}\left\lfloor\frac{ak}{cp}\right\rfloor\equiv \sum_{\substack{k=1\\\gcd(k,p)=1}}^{cp-1}k^{p-4}\left\lfloor\frac{ak}{cp}\right\rfloor\equiv \frac{a^3-a}{3}B_{p-3} \pmod{p}.
	\end{align}
	Therefore, for $p\ge 5$,
	\begin{align*}
	\sum_{\substack{k=1\\\gcd(k,p)=1}}^{cp-1} \frac{1}{k\cdot h(k)} \equiv \frac{1}{a}\cdot\frac{2c}{3}pB_{p-3}+ \frac{cp}{a^2} \cdot \frac{a^3-a}{3}B_{p-3}= \frac{c(a^2+1)}{3a}pB_{p-3} \pmod{p^2}.
	\end{align*}
	
	\textit{Case 2: $p= 3$}. We first show that for any positive integer $s$,
	\begin{align}\label{eq:3c-1}
	\sum_{\substack{k=1\\\gcd(k,3)=1}}^{3s-1}\frac{1}{k^2}\equiv -s \pmod{3^2}.
	\end{align}
	To see this, we write $s=9t+s_0$ where $t\ge 0$ and $1\le s_0\le 9$. Then
	\begin{align*}
	\sum_{\substack{k=1\\\gcd(k,3)=1}}^{3s-1}\frac{1}{k^2}&=\sum_{\substack{k_0=1\\\gcd(k_0,3)=1}}^{3s_0-1}\frac{1}{(27t+k_0)^2}+\sum_{t_0=0}^{t-1}\sum_{\substack{k_0=1\\\gcd(k_0,3)=1}}^{26}\frac{1}{(27t_0+k_0)^2}\\
	&\equiv \sum_{\substack{k_0=1\\\gcd(k_0,3)=1}}^{3s_0-1}\frac{1}{k_0^2}+\sum_{t_0=0}^{t-1}\sum_{\substack{k_0=1\\\gcd(k_0,3)=1}}^{26}\frac{1}{k_0^2} \pmod{3^2}.
	\end{align*}
	Also, a direct verification shows that
	\begin{align*}
	\sum_{\substack{k_0=1\\\gcd(k_0,3)=1}}^{26}\frac{1}{k_0^2} \equiv 0 \pmod{3^2}
	\end{align*}
	and for $1\le s_0\le 9$,
	\begin{align*}
	\sum_{\substack{k_0=1\\\gcd(k_0,3)=1}}^{3s_0-1}\frac{1}{k_0^2}\equiv -s_0 \pmod{3^2}.
	\end{align*}
	Thus, \eqref{eq:3c-1} follows. On the other hand, we deduce from \cite[Eq.~(1)]{Por1983} that
	\begin{align}\label{eq:floor-2}
	\sum_{\substack{k=1\\\gcd(k,3)=1}}^{3c-1}\frac{1}{k^3}\left\lfloor\frac{ak}{3c}\right\rfloor\equiv \sum_{\substack{k=1\\\gcd(k,3)=1}}^{3c-1}k\left\lfloor\frac{ak}{3c}\right\rfloor\equiv \frac{a^2-1}{12a} \pmod{3}.
	\end{align}
	Recalling that $a$ is not divisible by $3$, we have $3a^2\equiv 3 \pmod{3^2}$. Thus,
	\begin{align*}
	\sum_{\substack{k=1\\\gcd(k,3)=1}}^{3c-1} \frac{1}{k\cdot h(k)} \equiv -\frac{c}{a}+\frac{3c}{a^2}\cdot \frac{a^2-1}{12a}= -\frac{c(3a^2+1)}{4a^3}\equiv -\frac{c}{a^3} \pmod{3^2}.
	\end{align*}
	This confirms our result for $p=3$.
\end{proof}

\begin{lemma}\label{le:mod-cp}
	Let $p$ be an odd prime and $r$ be a positive integer. Let $u$ and $v$ be positive integers. For any positive integers $a$, $b$ and $c$ such that they are not divisible by $p$, and both $a$ and $b$ are coprime to $c$,
	\begin{enumerate}[label={\textup{(\alph*).}},leftmargin=*,labelsep=0cm,align=left]
		\item for $p=3$,
		\begin{align}
		\sum_{\substack{1\le i\le u c\cdot 3^r\\1\le j\le v c\cdot 3^r\\ai\equiv bj \ (\operatorname{mod}\;3c)\\\gcd(ij,3)=1}}\frac{1}{ij} \equiv \begin{cases}
		-uva^3b^3c & \text{if $r=1$}\\[8pt]
		-3^{2r-2}uv\big(a^3b^3c+3abc\big)& \text{if $r\ge 2$}
		\end{cases} \pmod{3^{2r}};
		\end{align}
		
		\item for $p\ge 5$,
		\begin{align}
		\sum_{\substack{1\le i\le u cp^r\\1\le j\le v cp^r\\ai\equiv bj \ (\operatorname{mod}\;cp)\\\gcd(ij,p)=1}}\frac{1}{ij}\equiv p^{2r-1}B_{p-3}\cdot \frac{uvc(a^2+b^2)}{3ab}\pmod{p^{2r}}.
		\end{align}
	\end{enumerate}
\end{lemma}

\begin{proof}
	Since $b$ is coprime to $c$ and $p$, we let $\overline{b}$ be such that $1\le \overline{b}\le cp^2-1$ and $\overline{b}b\equiv 1\pmod{cp^2}$. For each positive integer $i$ with $\gcd(k,p)=1$, we denote by $h(i)$ the unique integer such that $1\le h(i)\le cp-1$ and $h(i)\equiv a\overline{b}i \pmod{cp}$. Notice that
	\begin{align*}
	h(i)=-cp\left\lfloor\frac{a\overline{b}i}{cp}\right\rfloor+a\overline{b}i.
	\end{align*}
	For convenience, we write
	\begin{align*}
	\ell(i):=\left\lfloor\frac{a\overline{b}i}{cp}\right\rfloor.
	\end{align*}
	Let $m=\phi(p^2)-1=p^2-p-1$. It follows from Euler's theorem that
	\begin{align*}
	\sum_{\substack{1\le i\le u cp\\1\le j\le v cp\\ai\equiv bj \ (\operatorname{mod}\;cp)\\\gcd(ij,p)=1}}\frac{1}{ij}&=\sum_{\substack{i=1\\\gcd(i,p)=1}}^{u cp-1}\sum_{L=0}^{v-1}\frac{1}{i(cpL+h(i))}\\
	&\equiv \sum_{\substack{i=1\\\gcd(i,p)=1}}^{u cp-1}\sum_{L=0}^{v-1}\frac{(cpL-cp\ell(i)+a\overline{b}i)^m}{i}\\
	&\equiv \sum_{\substack{i=1\\\gcd(i,p)=1}}^{u cp-1}\sum_{L=0}^{v-1}\frac{(a\overline{b}i)^m-c(a\overline{b}i)^{m-1}(L-\ell(i))p}{i}\\
	&\equiv \frac{v}{a\overline{b}}\sum_{\substack{i=1\\\gcd(i,p)=1}}^{u cp-1}\frac{1}{i^2}-\frac{v(v-1)cp}{2a^2\overline{b}^2}\sum_{\substack{i=1\\\gcd(i,p)=1}}^{u cp-1}\frac{1}{i^3}\\
	&\quad+\frac{vcp}{a^2\overline{b}^2}\sum_{\substack{i=1\\\gcd(i,p)=1}}^{u cp-1}\frac{1}{i^3}\left\lfloor\frac{a\overline{b}i}{cp}\right\rfloor \pmod{p^2}.
	\end{align*}
	We have two cases:
	
	\textit{Case 1: $p\ge 5$}. First, by \cite[Theorem 1.1]{Hon2000},
	\begin{align*}
	\sum_{\substack{i=1\\\gcd(i,p)=1}}^{u cp-1}\frac{1}{i^2} \equiv \frac{2uc}{3}pB_{p-3} \pmod{p^2}.
	\end{align*}
	Also,
	\begin{align*}
	\sum_{\substack{i=1\\\gcd(i,p)=1}}^{u cp-1}\frac{1}{i^3}=\frac{1}{2}\sum_{\substack{i=1\\\gcd(i,p)=1}}^{u cp-1}\left(\frac{1}{i^3}+\frac{1}{(ucp-i)^3}\right)\equiv 0 \pmod{p}.
	\end{align*}
	Finally,
	\begin{align*}
	\sum_{\substack{i=1\\\gcd(i,p)=1}}^{u cp-1}\frac{1}{i^3}\left\lfloor\frac{a\overline{b}i}{cp}\right\rfloor&=\sum_{w=0}^{u-1}\sum_{k=1}^{cp-1}\frac{1}{(wcp+k)^3}\left(\left\lfloor\frac{a\overline{b}k}{cp}\right\rfloor+wa\overline{b}\right)\\
	&\equiv u\sum_{k=1}^{cp-1} \frac{1}{k^3}\left\lfloor\frac{a\overline{b}k}{cp}\right\rfloor+a\overline{b}\sum_{w=0}^{u-1} w \sum_{k=1}^{cp-1}\frac{1}{k^3}\\
	&\equiv u\sum_{k=1}^{cp-1} \frac{1}{k^3}\left\lfloor\frac{a\overline{b}k}{cp}\right\rfloor\\
	&\equiv\frac{u(a^3\overline{b}^3-a\overline{b})}{3}B_{p-3} \pmod{p},
	\end{align*}
	where we make use of \eqref{eq:floor}. It follows that for $p\ge 5$,
	\begin{align*}
	\sum_{\substack{1\le i\le u cp\\1\le j\le v cp\\ai\equiv bj \ (\operatorname{mod}\;cp)\\\gcd(ij,p)=1}}\frac{1}{ij}&\equiv\frac{v}{a\overline{b}}\cdot \frac{2uc}{3}pB_{p-3} + \frac{vcp}{a^2\overline{b}^2}\cdot \frac{u(a^3\overline{b}^3-a\overline{b})}{3}B_{p-3}\\
	&\equiv pB_{p-3}\cdot \frac{uvc(a^2+b^2)}{3ab}\pmod{p^{2}}.
	\end{align*}
	
	\textit{Case 2: $p=3$}. We know from \eqref{eq:3c-1} that
	\begin{align*}
	\sum_{\substack{k=1\\\gcd(k,3)=1}}^{3uc-1}\frac{1}{k^2}\equiv -uc \pmod{3^2}.
	\end{align*}
	Also,
	\begin{align*}
	\sum_{\substack{k=1\\\gcd(k,3)=1}}^{3uc-1}\frac{1}{k^3}\equiv 0 \pmod{3}.
	\end{align*}
	Finally,
	\begin{align*}
	\sum_{\substack{i=1\\\gcd(i,p)=1}}^{3u c-1}\frac{1}{i^3}\left\lfloor\frac{a\overline{b}i}{cp}\right\rfloor\equiv u\sum_{k=1}^{3c-1} \frac{1}{k^3}\left\lfloor\frac{a\overline{b}k}{cp}\right\rfloor \equiv u\cdot \frac{a^2\overline{b}^2-1}{12a\overline{b}} \pmod{3},
	\end{align*}
	where \eqref{eq:floor-2} is applied. Therefore,
	\begin{align*}
	\sum_{\substack{1\le i\le 3u c\\1\le j\le 3v c\\ai\equiv bj \ (\operatorname{mod}\;3c)\\\gcd(ij,3)=1}}\frac{1}{ij}&\equiv \frac{v}{a\overline{b}}\cdot (-uc) + \frac{3vc}{a^2\overline{b}^2}\cdot u\cdot \frac{a^2\overline{b}^2-1}{12a\overline{b}}\\
	&= -\frac{uvc(3a^2\overline{b}^2+1)}{4a^3\overline{b}^3}\equiv -\frac{uvc}{a^3\overline{b}^3}\equiv -uva^3b^3c\pmod{3^2}.
	\end{align*}
	The above confirms our result for $r=1$.
	
	Now, we assume $r\ge 2$. Apart from the notation defined at the beginning of this proof, we also require the following. Since $c$ is coprime to $p$, we let $\overline{c}$ be such that $1\le \overline{c}\le p^{2r}-1$ and $\overline{c}c\equiv 1\pmod{p^{2r}}$. Then
	\begin{align*}
	\sum_{\substack{1\le i\le u cp^r\\1\le j\le v cp^r\\ai\equiv bj \ (\operatorname{mod}\;cp)\\\gcd(ij,p)=1}}\frac{1}{ij}&=\sum_{\substack{k=1\\\gcd(k,p)=1}}^{cp-1}\sum_{\substack{1\le i\le u cp^r\\i\equiv k \ (\operatorname{mod}\;cp)}}\frac{1}{i}\sum_{\substack{1\le j\le v cp^r\\j\equiv a\overline{b}k \ (\operatorname{mod}\;cp)}}\frac{1}{j}\\
	&=\sum_{\substack{k=1\\\gcd(k,p)=1}}^{cp-1}\sum_{I=0}^{up^{r-1}-1}\frac{1}{cpI+k}\sum_{J=0}^{vp^{r-1}-1}\frac{1}{cpJ+h(k)}\\
	&=\sum_{\substack{k=1\\\gcd(k,p)=1}}^{cp-1}\frac{1}{c^2}\sum_{I=0}^{up^{r-1}-1}\frac{1}{pI+\frac{k}{c}}\sum_{J=0}^{vp^{r-1}-1}\frac{1}{pJ+\frac{h(k)}{c}}\\
	&\equiv \sum_{\substack{k=1\\\gcd(k,p)=1}}^{cp-1}\frac{1}{c^2}\, U(k\overline{c};u,p^r)\,U(h(k)\overline{c};v,p^r) \pmod{p^{2r}}.
	\end{align*}
	We have two cases:
	
	\textit{Case 1: $p\ge 5$}. Applying Lemma \ref{le:U}(d) yields
	\begin{align*}
	\sum_{\substack{1\le i\le u cp^r\\1\le j\le v cp^r\\ai\equiv bj \ (\operatorname{mod}\;cp)\\\gcd(ij,p)=1}}\frac{1}{ij}&\equiv \sum_{\substack{k=1\\\gcd(k,p)=1}}^{cp-1}\frac{1}{c^2}\cdot\frac{uv}{k\cdot h(k)\cdot \overline{c}^2}\cdot p^{2r-2}\\
	&\quad+\frac{1}{2}\sum_{\substack{k=1\\\gcd(k,p)=1}}^{cp-1}\frac{1}{c^2}\cdot\left(\frac{uv}{k\cdot h^2(k)\cdot \overline{c}^3}+\frac{uv}{k^2\cdot h(k)\cdot \overline{c}^3}\right)\cdot p^{2r-1}\\
	&\equiv uv \sum_{\substack{k=1\\\gcd(k,p)=1}}^{cp-1} \frac{1}{k\cdot h(k)}\cdot p^{2r-2}\\
	&\quad+\frac{uvc}{2}\sum_{\substack{k=1\\\gcd(k,p)=1}}^{cp-1}\left(\frac{1}{k\cdot h^2(k)}+\frac{1}{k^2\cdot h(k)}\right)\cdot p^{2r-1} \pmod{p^{2r}}.
	\end{align*}
	By Lemma \ref{le:h(k)},
	\begin{align*}
	\sum_{\substack{k=1\\\gcd(k,p)=1}}^{cp-1} \frac{1}{k\cdot h(k)}\equiv \frac{c(a^2\overline{b}^2+1)}{3a\overline{b}}pB_{p-3}\equiv \frac{c(a^2+b^2)}{3ab}pB_{p-3}\pmod{p^{2r}}.
	\end{align*}
	Also, recalling that $h(k)\equiv a\overline{b}k \pmod{cp}$, we have
	\begin{align*}
	\sum_{\substack{k=1\\\gcd(k,p)=1}}^{cp-1}\left(\frac{1}{k\cdot h^2(k)}+\frac{1}{k^2\cdot h(k)}\right)&\equiv \sum_{\substack{k=1\\\gcd(k,p)=1}}^{cp-1}\left(\frac{1}{a^2\overline{b}^2k^3}+\frac{1}{a\overline{b}k^3}\right)\\
	&=\left(\frac{1}{a^2\overline{b}^2}+\frac{1}{a\overline{b}}\right)\sum_{\substack{k=1\\\gcd(k,p)=1}}^{cp-1}\frac{1}{k^3}\\
%	&=\frac{1}{2}\left(\frac{1}{a^2\overline{b}^2}+\frac{1}{a\overline{b}}\right)\sum_{\substack{k=1\\\gcd(k,p)=1}}^{cp-1}\left(\frac{1}{k^3}+\frac{1}{(cp-k)^3}\right)\\
	&\equiv 0 \pmod{p}.
	\end{align*}
	Therefore, for $p\ge 5$,
	\begin{align*}
	\sum_{\substack{1\le i\le u cp^r\\1\le j\le v cp^r\\ai\equiv bj \ (\operatorname{mod}\;cp)\\\gcd(ij,p)=1}}\frac{1}{ij}\equiv p^{2r-1}B_{p-3}\cdot \frac{uvc(a^2+b^2)}{3ab}\pmod{p^{2r}}.
	\end{align*}
	
	\textit{Case 2: $p=3$}. Applying Lemma \ref{le:U}(d) yields
	\begin{align*}
	\sum_{\substack{1\le i\le u c\cdot 3^r\\1\le j\le v c\cdot 3^r\\ai\equiv bj \ (\operatorname{mod}\;3c)\\\gcd(ij,3)=1}}\frac{1}{ij}&\equiv uv \sum_{\substack{k=1\\\gcd(k,3)=1}}^{3c-1} \frac{1}{k\cdot h(k)}\cdot 3^{2r-2}\\
	&\quad+\frac{uvc}{2}\sum_{\substack{k=1\\\gcd(k,3)=1}}^{3c-1}\left(\frac{1}{k\cdot h^2(k)}+\frac{1}{k^2\cdot h(k)}\right)\cdot 3^{2r-1}\\
	&\quad+\frac{uvc^2}{2}\sum_{\substack{k=1\\\gcd(k,3)=1}}^{3c-1}\left(\frac{1}{k\cdot h^3(k)}+\frac{1}{k^3\cdot h(k)}\right)\cdot 3^{2r-1} \pmod{3^{2r}}.
	\end{align*}
	By Lemma \ref{le:h(k)} and the fact that $x^6\equiv 1 \pmod{3^2}$ for any $x$ not divisible by $3$,
	\begin{align*}
	\sum_{\substack{k=1\\\gcd(k,3)=1}}^{3c-1} \frac{1}{k\cdot h(k)}\equiv -\frac{c}{a^3\overline{b}^3}\equiv - \frac{cb^3}{a^3}\equiv -a^3b^3c \pmod{3^{2}}.
	\end{align*}
	Also,
	\begin{align*}
	\sum_{\substack{k=1\\\gcd(k,3)=1}}^{3c-1}\left(\frac{1}{k\cdot h^2(k)}+\frac{1}{k^2\cdot h(k)}\right)\equiv \left(\frac{1}{a^2\overline{b}^2}+\frac{1}{a\overline{b}}\right)\sum_{\substack{k=1\\\gcd(k,3)=1}}^{3c-1}\frac{1}{k^3}\equiv 0 \pmod{3}.
	\end{align*}
	Further,
	\begin{align*}
	\sum_{\substack{k=1\\\gcd(k,3)=1}}^{3c-1}\left(\frac{1}{k\cdot h^3(k)}+\frac{1}{k^3\cdot h(k)}\right)\equiv \left(\frac{1}{a^3\overline{b}^3}+\frac{1}{a\overline{b}}\right)\sum_{\substack{k=1\\\gcd(k,3)=1}}^{3c-1}\frac{1}{k^4} \pmod{3}.
	\end{align*}
	Similar to the proof of \eqref{eq:3c-1}, we have
	\begin{align*}
	\sum_{\substack{k=1\\\gcd(k,3)=1}}^{3c-1}\frac{1}{k^4}\equiv -c \pmod{3}.
	\end{align*}
	Recall also the fact that $x^2\equiv 1 \pmod{3}$ for any $x$ not divisible by $3$. Thus,
	\begin{align*}
	\sum_{\substack{k=1\\\gcd(k,3)=1}}^{3c-1}\left(\frac{1}{k\cdot h^3(k)}+\frac{1}{k^3\cdot h(k)}\right)\equiv \left(\frac{1}{a^3\overline{b}^3}+\frac{1}{a\overline{b}}\right)\cdot (-c)\equiv -2abc^{-1} \pmod{3}.
	\end{align*}
	We conclude that
	\begin{align*}
	\sum_{\substack{1\le i\le u c\cdot 3^r\\1\le j\le v c\cdot 3^r\\ai\equiv bj \ (\operatorname{mod}\;3c)\\\gcd(ij,3)=1}}\frac{1}{ij}&\equiv 3^{2r-2}uv\cdot (-a^3b^3c)+\frac{3^{2r-1}uvc^2}{2}\cdot (-2abc^{-1})\\
	&= -3^{2r-2}uv(a^3b^3c+3abc) \pmod{3^{2r}}.
	\end{align*}
	This completes the proof.
\end{proof}

\section{Theorem \ref{th:main} for $n=p^r$}

The object of this section is Theorem \ref{th:main} when $n$ is an odd prime power.

\begin{theorem}\label{th:ini}
	Let $a_1$, $a_2$ and $a_3$ be positive integers and let $A$ be a positive common multiple of $a_1$, $a_2$ and $a_3$. Let $g_1=\gcd(a_2,a_3)$, $g_2=\gcd(a_3,a_1)$, $g_3=\gcd(a_1,a_2)$ and $g=\gcd(a_1,a_2,a_3)$. Then for any odd prime $p$ not dividing $a_1$, $a_2$ and $a_3$, and any positive integer $r$,
	\begin{align}\label{eq:ini}
	\sum_{\substack{i,j,k\ge 1\\\gcd(ijk,p)=1\\a_1 i+a_2 j+a_3 k=Ap^r}}\frac{1}{ijk}\equiv -2p^{r-1}B_{p-3}\cdot \frac{Ag^3}{3}\left(\frac{1}{a_1^2 g_1^2}+\frac{1}{a_2^2 g_2^2}+\frac{1}{a_3^2 g_3^2}\right) \pmod{p^r}.
	\end{align}
\end{theorem}

Our starting point is the following specialization.

\begin{theorem}\label{th:ini-coprime}
	Let $c_1$, $c_2$ and $c_3$ be positive integers that are pairwise coprime and let $C$ be a positive common multiple of $c_1$, $c_2$ and $c_3$. Then for any odd prime $p$ not dividing $c_1$, $c_2$ and $c_3$, and any positive integer $r$,
	\begin{align}
	\sum_{\substack{i,j,k\ge 1\\\gcd(ijk,p)=1\\c_1 i+c_2 j+c_3 k=Cp^r}}\frac{1}{ijk}\equiv -2p^{r-1}B_{p-3}\cdot \frac{C}{3}\left(\frac{1}{c_1^2}+\frac{1}{c_2^2}+\frac{1}{c_3^2}\right) \pmod{p^r}.
	\end{align}
\end{theorem}

\begin{proof}
	We have
	\begin{align}\label{eq:trans-1}
	\sum_{\substack{i,j,k\ge 1\\\gcd(ijk,p)=1\\c_1 i+c_2 j+c_3 k=Cp^r}}\frac{1}{ijk} = \frac{1}{Cp^r}\sum_{\substack{i,j,k\ge 1\\\gcd(ijk,p)=1\\c_1 i+c_2 j+c_3 k=Cp^r}}\frac{c_1 i+c_2 j+c_3 k}{ijk}.
	\end{align}
	Notice that
	\begin{align*}
	\sum_{\substack{i,j,k\ge 1\\\gcd(ijk,p)=1\\c_1 i+c_2 j+c_3 k=Cp^r}}\frac{c_1 i}{ijk}&=\sum_{\substack{1\le \ell\le Cp^r/c_1\\\gcd(\ell,p)=1}}\sum_{\substack{1\le j\le Cp^r/c_2\\1\le k\le Cp^r/c_3\\c_2j+c_3k=c_1\ell\\\gcd(jk,p)=1}}\frac{c_1}{jk}\\
	&=\sum_{\substack{1\le \ell\le Cp^r/c_1\\\gcd(\ell,p)=1}}\frac{1}{\ell}\sum_{\substack{1\le j\le Cp^r/c_2\\1\le k\le Cp^r/c_3\\c_2j+c_3k=c_1\ell\\\gcd(jk,p)=1}}\frac{c_2j+c_3k}{jk}\\
	&=c_3\sum_{\substack{1\le \ell\le Cp^r/c_1\\1\le j\le Cp^r/c_2\\\gcd(\ell j,p)=1\\c_1\ell>c_2j\\c_1\ell\equiv c_2 j\ (\operatorname{mod}\;c_3)\\p\,\nmid\, c_1\ell-c_2 j}}\frac{1}{\ell j}+c_2\sum_{\substack{1\le \ell\le Cp^r/c_1\\1\le k\le Cp^r/c_3\\\gcd(\ell k,p)=1\\c_1\ell>c_3k\\c_1\ell\equiv c_3 k\ (\operatorname{mod}\;c_2)\\p\,\nmid\, c_1\ell-c_3 k}}\frac{1}{\ell k}.
	\end{align*}
	Applying the same argument to the remaining two terms on the right hand side of \eqref{eq:trans-1} yields
	\begin{align}\label{eq:trans-2}
	\sum_{\substack{i,j,k\ge 1\\\gcd(ijk,p)=1\\c_1 i+c_2 j+c_3 k=Cp^r}}\frac{1}{ijk} &= \frac{c_3}{Cp^r}\sum_{\substack{1\le i\le Cp^r/c_1\\1\le j\le Cp^r/c_2\\\gcd(i j,p)=1\\c_1 i \ne c_2j\\c_1 i\equiv c_2 j\ (\operatorname{mod}\;c_3)\\p\,\nmid\, c_1 i-c_2 j}}\frac{1}{i j} +\frac{c_2}{Cp^r}\sum_{\substack{1\le k\le Cp^r/c_3\\1\le i\le Cp^r/c_1\\\gcd(k i,p)=1\\c_3 k\ne c_1 i\\c_3 k\equiv c_1 i\ (\operatorname{mod}\;c_2)\\p\,\nmid\, c_3 k-c_1 i}}\frac{1}{ki}\notag\\
	&\quad+\frac{c_1}{Cp^r}\sum_{\substack{1\le j\le Cp^r/c_2\\1\le k\le Cp^r/c_3\\\gcd(j k,p)=1\\c_2j\ne c_3k\\c_2j\equiv c_3k\ (\operatorname{mod}\;c_1)\\p\,\nmid\, c_2j-c_3k}}\frac{1}{j k}.
	\end{align}
	
	Next, we notice that
	\begin{align}\label{eq:ij}
	\sum_{\substack{1\le i\le Cp^r/c_1\\1\le j\le Cp^r/c_2\\\gcd(i j,p)=1\\c_1 i \ne c_2j\\c_1 i\equiv c_2 j\ (\operatorname{mod}\;c_3)\\p\,\nmid\, c_1 i-c_2 j}}\frac{1}{i j} = \sum_{\substack{1\le i\le Cp^r/c_1\\1\le j\le Cp^r/c_2\\c_1 i\equiv c_2 j\ (\operatorname{mod}\;c_3)\\\gcd(i j,p)=1}}\frac{1}{i j}-\sum_{\substack{1\le i\le Cp^r/c_1\\1\le j\le Cp^r/c_2\\c_1 i\equiv c_2 j\ (\operatorname{mod}\;c_3p)\\\gcd(i j,p)=1}}\frac{1}{i j}.
	\end{align}
	Applying Lemmas \ref{le:mod-c} and \ref{le:mod-cp} with $a=c_1$, $b=c_2$, $c=c_3$, $u=C/(c_1c_3)$ and $v=C/(c_2c_3)$, we find that if $p\ge 5$,
	\begin{align*}
	\operatorname{LHS}\eqref{eq:ij}\equiv -p^{2r-1}B_{p-3}\cdot \frac{C^2}{3c_3}\left(\frac{1}{c_1^2}+\frac{1}{c_2^2}\right)\pmod{p^{2r}}
	\end{align*}
	and if $p=3$,
	\begin{align*}
	\operatorname{LHS}\eqref{eq:ij}\equiv 3^{2r-2}\cdot \frac{C^2}{c_3}\big(c_1^2c_2^2+3\delta(r)\big) \pmod{3^{2r}}
	\end{align*}
	where $\delta(r)=0$ if $r=1$ and $1$ if $r\ge 2$.
	
	Applying the same argument to the remaining two terms on the right hand side of \eqref{eq:trans-2} immediately gives \eqref{eq:ini} for $p\ge 5$.
	
	For $p=3$, however, we deduce
	\begin{align*}
	\sum_{\substack{i,j,k\ge 1\\\gcd(ijk,3)=1\\c_1 i+c_2 j+c_3 k=C\cdot 3^r}}\frac{1}{ijk}\equiv 3^{r-2}\cdot C (c_1^2c_2^2+c_2^2c_3^2+c_3^2c_1^2) \pmod{3^r}.
	\end{align*}
	Recalling that $c_1$, $c_2$ and $c_3$ are not divisible by $3$, we have
	\begin{align*}
	c_1^2c_2^2+c_2^2c_3^2+c_3^2c_1^2 \equiv 0 \pmod{3}
	\end{align*}
	and
	\begin{align*}
	-\frac{2}{c_1^2c_2^2c_3^2}\equiv 1 \pmod{3}.
	\end{align*}
	Therefore,
	\begin{align*}
	\sum_{\substack{i,j,k\ge 1\\\gcd(ijk,3)=1\\c_1 i+c_2 j+c_3 k=C\cdot 3^r}}\frac{1}{ijk}\equiv -\frac{2}{c_1^2c_2^2c_3^2}\cdot 3^{r-2}\cdot C (c_1^2c_2^2+c_2^2c_3^2+c_3^2c_1^2) \pmod{3^r}.
	\end{align*}
	This establishes \eqref{eq:ini} for $p=3$.
\end{proof}

%Now, let $C$ be a positive common multiple of $c_1$, $c_2$ and $c_3$. Notice that we can write $C=mC_0$ for a certain positive integer $m$ since $c_1$, $c_2$ and $c_3$ are pairwise coprime. Then by Lemma \ref{le:multiple},
%\begin{align}\label{eq:main-coprime}
%\sum_{\substack{i,j,k\ge 1\\\gcd(ijk,n)=1\\c_1 i+c_2 j+c_3 k=Cp^r}}\frac{1}{ijk} &\equiv m \sum_{\substack{i,j,k\ge 1\\\gcd(ijk,p)=1\\c_1 i+c_2 j+c_3 k=C_0p^r}}\frac{1}{ijk}\notag\\
%&\equiv -2p^{r-1}B_{p-3}\cdot \frac{C}{3}\left(\frac{1}{c_1^2}+\frac{1}{c_2^2}+\frac{1}{c_3^2}\right) \pmod{p^r}.
%\end{align}

\begin{proof}[Proof of Theorem \ref{th:ini}]
	For convenience, we define
	$$b_1=\frac{a_1g}{g_2g_3},\quad b_2=\frac{a_2g}{g_3g_1},\quad b_3=\frac{a_3g}{g_1g_2}.$$
	It is straightforward to verify that $b_1$, $b_2$ and $b_3$ are positive integers that are pairwise coprime. Also, $Ag^2/(g_1g_2g_3)$ is a common multiple of $b_1$, $b_2$ and $b_3$.
	
	Now, notice that there is a bijection between
	\begin{align*}
	\left\{(i,j,k)\in\mathbb{Z}_{>0}:a_1i+a_2j+a_3k=Ap^r\ \text{and}\ \gcd(ijk,p)=1\right\}
	\end{align*}
	and
	\begin{align*}
	\left\{(I,J,K)\in\mathbb{Z}_{>0}:b_1I+b_2J+b_3K=\frac{Ag^2}{g_1g_2g_3}p^r\ \text{and}\ \gcd(IJK,p)=1\right\}
	\end{align*}
	given by
	\begin{align*}
	(I,J,K)=\left(\frac{gi}{g_1},\frac{gj}{g_2},\frac{gk}{g_3}\right).
	\end{align*}
	Thus,
	\begin{align*}
	\sum_{\substack{i,j,k\ge 1\\\gcd(ijk,p)=1\\a_1 i+a_2 j+a_3 k=Ap^r}}\frac{1}{ijk} = \frac{g^3}{g_1g_2g_3}\sum_{\substack{I,J,K\ge 1\\\gcd(IJK,p)=1\\b_1I+b_2J+b_3K=\frac{Ag^2}{g_1g_2g_3}p^r}}\frac{1}{IJK}.
	\end{align*}
	By Theorem \ref{th:ini-coprime},
	\begin{align*}
	\sum_{\substack{I,J,K\ge 1\\\gcd(IJK,p)=1\\b_1I+b_2J+b_3K=\frac{Ag^2}{g_1g_2g_3}p^r}}\frac{1}{IJK} \equiv -2p^{r-1}B_{p-3}\cdot \frac{Ag^2}{3g_1g_2g_3}\left(\frac{1}{b_1^2}+\frac{1}{b_2^2}+\frac{1}{b_3^2}\right) \pmod{p^r}.
	\end{align*}
	Combining the above two relations and inserting the expressions of $b_1$, $b_2$ and $b_3$, the desired result follows.
\end{proof}

\section{An auxiliary polynomial}

Throughout, for any polynomial or formal power series $f(x)$, we use the conventional notation $[x^m]f(x)$ to denote the coefficient of $x^m$ in $f(x)$ for any nonnegative integer $m$.

We introduce a polynomial for positive integers $M$ and $N$:
\begin{align}\label{eq:F-def}
f(x;M,N):=\sum_{\substack{k=1\\\gcd(k,N)=1}}^{MN}\frac{x^k}{k}.
\end{align}
Notice that
\begin{align}\label{eq:key}
\sum_{\substack{i,j,k\ge 1\\\gcd(ijk,n)=1\\a_1 i+a_2 j+a_3 k=An}}\frac{1}{ijk} = [x^{An}]\prod_{i=1}^3 f(x^{a_i};A/a_i,n).
\end{align}
This relation will be repeatedly used in the sequel.

\begin{lemma}\label{le:2AN}
	Let $a_1$, $a_2$ and $a_3$ be positive integers and let $A$ be a positive common multiple of $a_1$, $a_2$ and $a_3$. For any positive integer $N$, the polynomial
	$$\prod_{i=1}^3 f(x^{a_i};A/a_i,N)$$
	has no constant term and its degree is smaller than $3AN$. Further,
	\begin{align}\label{eq:2AN}
	[x^{2AN}]\prod_{i=1}^3 f(x^{a_i};A/a_i,N)\equiv -[x^{AN}]\prod_{i=1}^3 f(x^{a_i};A/a_i,N) \pmod{N}.
	\end{align}
\end{lemma}

\begin{proof}
	For convenience, we write $a_i'=A/a_i$ for $i=1,2,3$. The first part comes from the fact that for each $i$, $f(x^{a_i};a_i',N)$ is a polynomial of degree smaller than $AN$ whose constant term vanishes. For the second part, we have, modulo $N$,
	\begin{align*}
	[x^{2AN}]\prod_{i=1}^3 f(x^{a_i};a_i',N) &= \sum_{\substack{1\le i\le a_1'N\\1\le j\le a_2'N\\1\le k\le a_3'N\\\gcd(ijk,N)=1\\a_1 i+a_2 j+a_3 k=2AN}}\frac{1}{ijk}\\
	&=\sum_{\substack{i,j,k\ge 1\\\gcd(ijk,N)=1\\a_1 i+a_2 j+a_3 k=AN}}\frac{1}{(a_1'N-i)(a_2'N-j)(a_3'N-k)}\\
	&\equiv -\sum_{\substack{i,j,k\ge 1\\\gcd(ijk,N)=1\\a_1 i+a_2 j+a_3 k=AN}}\frac{1}{ijk},
	\end{align*}
	which gives \eqref{eq:2AN} by recalling \eqref{eq:key}.
\end{proof}

Also, we require the following property of the polynomial $f$.

\begin{lemma}\label{le:mod-key}
	Let $L$, $M$ and $N$ be positive integers. Then
	\begin{align}
	f(x;LM,N)\equiv \frac{1-x^{LMN}}{1-x^{MN}}f(x;M,N) \pmod{N}.
	\end{align}
\end{lemma}

\begin{proof}
	We have, modulo $N$,
	\begin{align*}
	f(x;LM,N)&=\sum_{\substack{k=1\\\gcd(k,N)=1}}^{LMN}\frac{x^k}{k}=\sum_{\ell=0}^{L-1}\sum_{\substack{k=1\\\gcd(k,N)=1}}^{MN}\frac{x^{\ell MN+k}}{\ell MN+k}\\
	&\equiv \sum_{\ell=0}^{L-1} x^{\ell MN} \sum_{\substack{k=1\\\gcd(k,N)=1}}^{MN}\frac{x^{k}}{k}=\frac{1-x^{LMN}}{1-x^{MN}}f(x;M,N).
	\end{align*}
	We therefore arrive at the desired result.
\end{proof}

\section{A proof by induction}

Now, we are ready to prove Theorem \ref{th:main} by induction on the number of prime factors of $n$.

Let $p$ be an odd prime not dividing $a_1$, $a_2$ and $a_3$. We assume that
\begin{align}\label{eq:ind-base}
[x^{An}]\prod_{i=1}^3 f(x^{a_i};A/a_i,n)&=\sum_{\substack{i,j,k\ge 1\\\gcd(ijk,n)=1\\a_1 i+a_2 j+a_3 k=An}}\frac{1}{ijk}\notag\\
&\equiv -2B_{p-3}\cdot \frac{n}{p}\cdot \frac{Ag^3}{3}\left(\frac{1}{a_1^2 g_1^2}+\frac{1}{a_2^2 g_2^2}+\frac{1}{a_3^2 g_3^2}\right)\notag\\
&\quad\times \prod_{\substack{\text{prime $q\mid n$}\\q\ne p}}\left(1-\frac{2}{q}\right)\left(1-\frac{1}{q^3}\right) \pmod{p^r}
\end{align}
holds true for some $n$ with $p^r\parallel n$. We have established the case where $n=p^r$ in Theorem \ref{th:ini}.

Our object is to show that for any prime $q$ not dividing $n$, and any positive integer $s$,
\begin{align}\label{eq:induction}
\sum_{\substack{i,j,k\ge 1\\\gcd(ijk,q^sn)=1\\a_1 i+a_2 j+a_3 k=Aq^sn}}\frac{1}{ijk}\equiv q^s\left(1-\frac{2}{q}\right)\left(1-\frac{1}{q^3}\right)\sum_{\substack{i,j,k\ge 1\\\gcd(ijk,n)=1\\a_1 i+a_2 j+a_3 k=An}}\frac{1}{ijk} \pmod{p^r}.
\end{align}
Then Theorem \ref{th:main} follows by induction.

\begin{proof}[Proof of \eqref{eq:induction}]
	Let $a_i'=A/a_i$ for $i=1,2,3$. By \eqref{eq:key},
	\begin{align*}
	\sum_{\substack{i,j,k\ge 1\\\gcd(ijk,q^sn)=1\\a_1 i+a_2 j+a_3 k=Aq^sn}}\frac{1}{ijk} = [x^{Aq^sn}]\prod_{i=1}^3 f(x^{a_i};a_i',q^sn).
	\end{align*}
	We then reformulate each $f(x^{a_i};a_i',q^sn)$ as follows. First, by \eqref{eq:F-def},
	\begin{align*}
	f(x^{a_i};a_i',q^sn) &= \sum_{\substack{k=1\\\gcd(k,q^sn)=1}}^{a_i'q^sn}\frac{x^{a_ik}}{k}\\
	&=\sum_{\substack{k=1\\\gcd(k,n)=1}}^{a_i'q^sn}\frac{x^{a_ik}}{k} - \sum_{\substack{k=1\\\gcd(k,n)=1}}^{a_i'q^{s-1}n}\frac{x^{a_iqk}}{qk}\\
	&=f(x^{a_i};a_i'q^s,n)-\frac{1}{q}f(x^{a_iq};a_i'q^{s-1},n).
	\end{align*} 
	Recalling that $p^r\mid n$, we apply Lemma \ref{le:mod-key} to $f(x^{a_i};a_i'q^s,n)$ with $(x,L,M,N)\mapsto (x^{a_i},q^{s-1},a_i'q,n)$ and to $f(x^{a_iq};a_i'q^{s-1},n)$ with $(x,L,M,N)\mapsto (x^{a_iq},q^{s-1},a_i',n)$, respectively. Thus,
	\begin{align}
	f(x^{a_i};a_i',q^sn) \equiv \frac{1-x^{Aq^sn}}{1-x^{Aqn}}\left(f(x^{a_i};a_i'q,n)-\frac{1}{q}f(x^{a_iq};a_i',n)\right) \pmod{p^r}.
	\end{align}
	It follows that
	\begin{align}
	&\sum_{\substack{i,j,k\ge 1\\\gcd(ijk,q^sn)=1\\a_1 i+a_2 j+a_3 k=Aq^sn}}\frac{1}{ijk}\notag\\
	&\qquad \equiv [x^{Aq^sn}]\left(\frac{1-x^{Aq^sn}}{1-x^{Aqn}}\right)^3\prod_{i=1}^3 \left(f(x^{a_i};a_i'q,n)-\frac{1}{q}f(x^{a_iq};a_i',n)\right)\pmod{p^r}.
	\end{align}
	
	Since we want to calculate the coefficient of $x^{Aq^s n}$ in the above while the cube on the right hand side is a polynomial of $x^{Aqn}$, we may expand the product on the right hand side into eight parts and apply Lemma \ref{le:2AN}. Thus,
	\begin{align*}
	\sum_{\substack{i,j,k\ge 1\\\gcd(ijk,q^sn)=1\\a_1 i+a_2 j+a_3 k=Aq^sn}}\frac{1}{ijk}&\equiv [x^{Aqn(q^{s-1}-1)}]\frac{1}{(1-x^{Aqn})^3}\cdot [x^{Aqn}]\sum_{j=1}^8 S_j(x)\\
	&\quad+[x^{Aqn(q^{s-1}-2)}]\frac{1}{(1-x^{Aqn})^3}\cdot [x^{2Aqn}]\sum_{j=1}^8 S_j(x)\\ &\equiv\left([x^{Aqn(q^{s-1}-1)}]\frac{1}{(1-x^{Aqn})^3}-[x^{Aqn(q^{s-1}-2)}]\frac{1}{(1-x^{Aqn})^3}\right)\\
	&\quad\times [x^{Aqn}]\sum_{j=1}^8 S_j(x) \pmod{p^r},
	\end{align*}
	where
	\begin{align*}
	S_1(x)&:=f(x^{a_1};a_1'q,n)f(x^{a_2};a_2'q,n)f(x^{a_3};a_3'q,n),\\
	S_2(x)&:=-\frac{1}{q}f(x^{a_1q};a_1',n)f(x^{a_2};a_2'q,n)f(x^{a_3};a_3'q,n),\\
	S_3(x)&:=-\frac{1}{q}f(x^{a_1};a_1'q,n)f(x^{a_2q};a_2',n)f(x^{a_3};a_3'q,n),\\
	S_4(x)&:=-\frac{1}{q}f(x^{a_1};a_1'q,n)f(x^{a_2};a_2'q,n)f(x^{a_3q};a_3',n),\\
	S_5(x)&:=\frac{1}{q^2}f(x^{a_1q};a_1',n)f(x^{a_2q};a_2',n)f(x^{a_3};a_3'q,n),\\
	S_6(x)&:=\frac{1}{q^2}f(x^{a_1q};a_1',n)f(x^{a_2};a_2'q,n)f(x^{a_3q};a_3',n),\\
	S_7(x)&:=\frac{1}{q^2}f(x^{a_1};a_1'q,n)f(x^{a_2q};a_2',n)f(x^{a_3q};a_3',n),\\
	S_8(x)&:=-\frac{1}{q^3}f(x^{a_1q};a_1',n)f(x^{a_2q};a_2',n)f(x^{a_3q};a_3',n).
	\end{align*}
	Notice also that for any nonnegative integer $m$,
	\begin{align*}
	[z^m]\frac{1}{(1-z)^3}&=[z^m](1+z+z^2+z^3+\cdots)^3\notag\\
	&=\operatorname{card}\{(u,v,w)\in\mathbb{Z}_{\ge 0}^3:u+v+w=m\}\\
	&=\binom{m+2}{2}.
	\end{align*}
	Thus,
	\begin{align*}
	\sum_{\substack{i,j,k\ge 1\\\gcd(ijk,q^sn)=1\\a_1 i+a_2 j+a_3 k=Aq^sn}}\frac{1}{ijk}\equiv q^{s-1}\cdot [x^{Aqn}]\sum_{j=1}^8 S_j(x) \pmod{p^r}.
	\end{align*}
	
	If we write
	\begin{align*}
	\Xi:=-2B_{p-3}\cdot \frac{n}{p}\cdot\prod_{\substack{\text{prime $q_0\mid n$}\\q\ne p}}\left(1-\frac{2}{q_0}\right)\left(1-\frac{1}{q_0^3}\right),
	\end{align*}
	then applying \eqref{eq:ind-base} to each $S_j(x)$ gives
	\begin{align*}
	[x^{Aqn}]S_1(x)&\equiv\Xi\cdot\frac{Aqg^3}{3}\left(\frac{1}{a_1^2 g_1^2}+\frac{1}{a_2^2 g_2^2}+\frac{1}{a_3^2 g_3^2}\right),\\
	[x^{Aqn}]S_2(x)&\equiv-\Xi\cdot\frac{1}{q}\cdot\frac{Aqg^3}{3}\left(\frac{1}{a_1^2q^2 g_1^2}+\frac{1}{a_2^2 g_2^2}+\frac{1}{a_3^2 g_3^2}\right),\\
	[x^{Aqn}]S_3(x)&\equiv-\Xi\cdot\frac{1}{q}\cdot\frac{Aqg^3}{3}\left(\frac{1}{a_1^2 g_1^2}+\frac{1}{a_2^2q^2 g_2^2}+\frac{1}{a_3^2 g_3^2}\right),\\
	[x^{Aqn}]S_4(x)&\equiv-\Xi\cdot\frac{1}{q}\cdot\frac{Aqg^3}{3}\left(\frac{1}{a_1^2 g_1^2}+\frac{1}{a_2^2 g_2^2}+\frac{1}{a_3^2q^2 g_3^2}\right),\\
	[x^{Aqn}]S_5(x)&\equiv\Xi\cdot\frac{1}{q^2}\cdot\frac{Aqg^3}{3}\left(\frac{1}{a_1^2q^2 g_1^2}+\frac{1}{a_2^2q^2 g_2^2}+\frac{1}{a_3^2 g_3^2q^2}\right),\\
	[x^{Aqn}]S_6(x)&\equiv\Xi\cdot\frac{1}{q^2}\cdot\frac{Aqg^3}{3}\left(\frac{1}{a_1^2q^2 g_1^2}+\frac{1}{a_2^2 g_2^2q^2}+\frac{1}{a_3^2q^2 g_3^2}\right),\\
	[x^{Aqn}]S_7(x)&\equiv\Xi\cdot\frac{1}{q^2}\cdot\frac{Aqg^3}{3}\left(\frac{1}{a_1^2 g_1^2q^2}+\frac{1}{a_2^2q^2 g_2^2}+\frac{1}{a_3^2q^2 g_3^2}\right),\\
	[x^{Aqn}]S_8(x)&\equiv-\Xi\cdot\frac{1}{q^3}\cdot\frac{Aqg^3q^3}{3}\left(\frac{1}{a_1^2q^2 g_1^2q^2}+\frac{1}{a_2^2q^2 g_2^2q^2}+\frac{1}{a_3^2q^2 g_3^2q^2}\right),
	\end{align*}
	all modulo $p^r$. Thus,
	\begin{align*}
	\sum_{\substack{i,j,k\ge 1\\\gcd(ijk,q^sn)=1\\a_1 i+a_2 j+a_3 k=Aq^sn}}\frac{1}{ijk}&\equiv q^{s-1}\cdot [x^{Aqn}]\sum_{j=1}^8 S_j(x)\\
	&\equiv \Xi \cdot \frac{Ag^3}{3}\left(\frac{1}{a_1^2 g_1^2}+\frac{1}{a_2^2 g_2^2}+\frac{1}{a_3^2 g_3^2}\right)\\
	&\quad\times q^{s-1}\cdot\left(q-2-\frac{1}{q^2}+\frac{3}{q^3}-\frac{1}{q^3}\right)\\
	&=\Xi \cdot \frac{Ag^3}{3}\left(\frac{1}{a_1^2 g_1^2}+\frac{1}{a_2^2 g_2^2}+\frac{1}{a_3^2 g_3^2}\right)\\
	&\quad\times q^{s}\cdot\left(1-\frac{2}{q}\right)\left(1-\frac{1}{q^3}\right) \pmod{p^r}.
	\end{align*}
	Inserting the expression of $\Xi$ and recalling \eqref{eq:ind-base}, we arrive at \eqref{eq:induction}.
\end{proof}

\subsection*{Acknowledgements}

The author was supported by a Killam Postdoctoral Fellowship from the Killam Trusts.

\bibliographystyle{amsplain}

\end{document}